\documentclass{amsart}
\usepackage[utf8x]{inputenc}
\usepackage[unicode]{hyperref}
\usepackage{mathrsfs,amssymb}
\usepackage{esint}

\numberwithin{equation}{section}
\allowdisplaybreaks

\newtheorem{theorem}{Theorem}[section]
\newtheorem{proposition}[theorem]{Proposition}
\newtheorem{lemma}[theorem]{Lemma}

\newtheorem{remark}[theorem]{Remark}

\hypersetup{
breaklinks=true,
colorlinks=true,
linkcolor=blue,
citecolor=blue,
urlcolor=blue,
}

\newcommand{\abs}[1]{{\left\lvert#1\right\rvert}}
\newcommand{\nabs}[1]{{\lvert#1\rvert}}
\newcommand{\br}[1]{\ensuremath{\left(#1\right)}}
\DeclareMathOperator{\dist}{dist}

\DeclareMathOperator{\identity}{id}
\DeclareMathOperator{\cof}{cof}

\newcommand{\norm}[1]{\ensuremath{\left\lVert#1\right\rVert}}

\newcommand{\sett}[2]{\ensuremath{\left\{#1\,\middle|\,#2\right\}}}

\newcommand{\sq}[1]{\ensuremath{\left[#1\right]}}

\renewcommand{\d}{\ensuremath{\,\mathrm{d}}}
\newcommand{\eps}{\ensuremath{\varepsilon}}
\newcommand{\N}{\ensuremath{\mathbb{N}}}
\newcommand{\R}{\ensuremath{\mathbb{R}}}
\newcommand{\Ecal}{\ensuremath{\mathcal E}}
\newcommand{\Dcal}{\ensuremath{\mathcal D}}
\newcommand{\Dcalt}{\ensuremath{\widetilde{\mathcal D}}}

\title{Nonlinear elasticity with vanishing nonlocal self-repulsion}
\author{Stefan Kr\"omer}
\address[Stefan Kr\"omer]{The Czech Academy of Sciences, Institute of Information Theory and Automation (\'{U}TIA), Pod vod\'{a}renskou v\v{e}\v{z}\'{\i}~4, 182~08~Praha~8, Czech Republic}
\email{skroemer@utia.cas.cz}
\author{Philipp Reiter}
\address[Philipp Reiter]{Chemnitz University of Technology,
Faculty of Mathematics, 09107 Chemnitz, Germany}
\email{reiter@math.tu-chemnitz.de}

\begin{document}
\begin{abstract}
We prove that that for nonlinear elastic energies with 
strong enough energetic control of the outer distortion of 
admissible deformations, 
almost everywhere global invertibility as constraint can be obtained 
in the $\Gamma$-limit 
of the elastic energy with an added nonlocal self-repulsion term with asymptocially vanishing coefficient.
The self-repulsion term considered here formally coincides with a Sobolev--Slobodecki\u\i\ seminorm of the inverse deformation.
Variants near the boundary or on the surface of the domain are also studied.
\end{abstract}
\maketitle

\tableofcontents

\section{Introduction}\label{sec:intro}

A basic requirement for many models of deformable solids is that they should prevent interpenetration of mass.
In context of hyperelasticity, i.e., nonlinear elasticity fully determined by a stored elastic energy function (see, e.g., \cite{Ba77a,Cia88B} for an introduction), 
this is ensured by a strong local resistance to compression 
built into the energy density, which in particular prevents local
change of orientation, combined with a constraint preventing global self-penetration, usually
the Ciarlet--Ne\v{c}as condition \cite{CiaNe87a}, see \eqref{eq:CNC} below.

In this article, we study the approximation of the latter by augmenting 
the local elastic energy with a nonlocal functional with self-repulsive properties, formally corresponding to suitable Sobolev--Slobodecki\u\i\ seminorms of the inverse deformation.
While all results presented here are purely analytical, our motivation is mainly numerical, related to the fact that the Ciarlet--Ne\v{c}as condition
is hard to handle numerically in such a way that the 
algorithm maintains an acceptable computational cost while still provably converging.
In particular, there is still no known projection onto the Ciarlet--Ne\v{c}as condition
which is rigorous with acceptable computational cost, see \cite{AigLi2013a} 
for some partial results.
There is a well-known straightforward penalty term that rigorously reproduces the Ciarlet--Ne\v{c}as condition in the limit (see \cite{MieRou16a}, e.g.), but it is hard to implement, non-smooth and computationally very expensive as a double integral on the full domain. Recent results on more practical rigorous approximation of the Ciarlet--Ne\v{c}as condition via nonlocal penalty terms added to the elastic energy were obtained in \cite{KroeVa19a} 
and \cite{KroeVa22Pa}, but these require additional regularity of elastic deformations
which possibly interferes with the Lavrentiev phenomenon which is known to appear at least in particular nonlinear elastic models~\cite{FoHruMi03a}. 

Using the language of $\Gamma$-convergence (see \cite{Brai02B}, e.g.), we show that in combination with local nonlinear elastic energies, the self-repulsive terms studied here also provide a rigorous approximation of the Ciarlet--Ne\v{c}as condition without requiring regularity of deformations beyond what is naturally provided by the nonlinear elastic energy (Theorem~\ref{thm:main}).
In addition, these admit natural variants near or on the boundary (Theorems~\ref{thm:main2} and~\ref{thm:main3}), which are significantly cheaper to compute in practice. The latter crucially rely on a global
invertibility property of orientation preserving maps exploiting topological information on the boundary \cite{Kroe20a} 
(for related results see also \cite{Ba81a,HeMoCoOl21a}).

Our results here still do not cover the full range of hyperelastic energies which are known to be variationally well-posed, though. In fact, we require lower bounds on the energy density which are strong enough so that deformation maps with finite elastic energy are
automatically continuous, open and discrete, the latter two by the theory of functions of bounded distortion \cite{HeKo14B}.
In our proofs, this is essential so that all local regularity is controlled by the elastic energy, while the nonlocal self-repulsive term asymptotically only controls global self-contact.

For related results concerning self-avoiding curves and surfaces in more geometrical context with higher regularity, we refer to
\cite{BaReiRie18a,BaRei20a,BaRei21a,YuBraSchuKee21a} and references therein.

\subsection*{General assumptions}

Let $\Omega\subset\R^{d}$ be a bounded Lipschitz domain,
$d\ge2$, $p\in(d,\infty)$, $r>0$, $\eps\ge0$,
$q\ge 1$ and $s\geq 0$.
By $W_{+}^{1,p}(\Omega,\R^{d})$ we denote the set of all functions
$y\in W^{1,p}(\Omega,\R^{d})$ with $\det \nabla y(x)>0$ for a.e.\ $x\in\Omega$.

We consider an integral functional modeling the internal elastic energy of a 
deformation $y\in W_{+}^{1,p}(\Omega,\R^{d})$ of a nonlinear hyperelastic solid. For simplicity,
we restrict ourselves to the following ``neo-Hookian'' form given by
\[ \Ecal(y) = \int_{\Omega}\abs{\nabla y(x)}^{p}\d x + \int_{\Omega}\frac{\d x}{\br{\det \nabla y(x)}^{r}}. \]

We are interested in precluding deformations
corresponding to self-interpenetration of matter,
i.e., non-injective $y$.
Classically, the latter is imposed by adding the
\emph{Ciarlet--Ne\v cas condition} \cite{CiaNe87a} as a constraint:
\begin{align}\label{eq:CNC}\tag{CN}
 \int_{\Omega}\abs{\det \nabla y(x)}\d x= \abs{y(\Omega)}.
\end{align}
\begin{remark}\label{rem:CNC}
By the area formula, the inequality \textup{``$\ge$''} in~\eqref{eq:CNC} always holds true and \eqref{eq:CNC} is equivalent to a.e.~injectivity of $y$ provided that $\det\nabla y>0$ a.e. As the latter is usually given, \eqref{eq:CNC} can also be expressed in the more standard form
 \begin{align*}
  \int_{\Omega}\det \nabla y(x)\d x\le \abs{y(\Omega)}.
 \end{align*}
\end{remark}

As a step towards a possible (numerical) approximation of \eqref{eq:CNC},
we regularize $\Ecal$
by adding
a singular nonlocal contribution $\Dcal$.
Below, we will show that  \eqref{eq:CNC} automatically holds whenever $\Ecal(y)<\infty$ and $\Dcal(y)<\infty$ with various examples for~$\Dcal$, see Propositions~\ref{prop:CNC}, \ref{prop:CNC2} and \ref{prop:CNC3}.

The first such example for $\Dcal$ is given by
\[ 
\Dcal_{U}(y) = \iint_{U\times U}\frac{\abs{x-\tilde x}^{q}}{\abs{y(x)-y(\tilde x)}^{d
+sq}}\abs{\det \nabla y(x)}\abs{\det \nabla y(\tilde x)}
\d x\d\tilde x \]
where $U\subset\Omega$ is some open neighborhood of~$\partial\Omega$ in~$\Omega$
and suitable parameters $q\in[1,\infty)$, $s\in[0,1)$.
In particular, we can choose $U=\Omega$.
Transforming the integral and invoking~\cite[Prop.~2]{MR1942116} reveals that the integral is singular if $s\ge1$.

Formally, after a change of variables,
$\Dcal_{U}$ is the Sobolev--Slobodecki\u\i\
semi-norm of $y^{-1}$ in the space $W^{s,q}(y(U),\R^{d})$. 
As long as $sq\ge 0$, the functional $\Dcal_{U}$ effectively prevents self-interpenetration, i.e., a loss of injectivity of $y$, as shown in Proposition~\ref{prop:CNC} below.
To the best of our knowledge, variants of $\Dcal_{U}$ for curves, with such a purpose in mind, first appeared in a master thesis~\cite{unseld} supervised by Dziuk, and have subsequently been studied in another master thesis~\cite{hermes}.
The functional~$\Dcal_{U}$ can be interpreted to be a sort of ``relaxation'' of the bi-Lipschitz constant.
In this sense it is a rather weak quantity in comparison with similar concepts that have been introduced earlier~\cite{MR1098918,MR1692638}.

We also study the boundary variant
\[
\Dcalt_{\partial\Omega}(y) = \iint_{\partial\Omega\times\partial\Omega}\frac{\abs{x-\tilde x}^{q}}{\abs{y(x)-y(\tilde x)}^{d-1
+sq}}
\d A(x)\d A(\tilde x) \]
where $A$ denotes the $(d-1)$-dimensional Hausdorff measure.

We give a rigorous statement of this approximation by
establishing $\Gamma$-convergence which is the main result of this paper.
To this end, we consider
for $y\in W^{1,p}(\Omega,\R^{d})$, $\eps>0$, 
\[ E_{\eps} (y) = \begin{cases}
\Ecal(y) + \eps\Dcal(y) & \text{if }y\in W_{+}^{1,p}(\Omega,\R^{d}), \\
+\infty &\text{else}.
\end{cases}
 \]
We reserve the symbol $E_{0}$ for the $\Gamma$-limit
which will turn out to be
\[ E_{0}(y) :=
\begin{cases}
\Ecal(y) & \text{if $y\in W_{+}^{1,p}(\Omega,\R^{d})$ 
satisfies~\eqref{eq:CNC},} \\
+\infty  & \text{else.}
\end{cases} \]

\section{Preliminary results}\label{sec:pre}

Change of variables in quite general form is important for us throughout, in form of 
the following special case of the area formula due to Marcus and Mizel.
\begin{lemma}[cf.~{\cite[Theorem 2]{MaMi73a}}]\label{lem:transform}
 Let $y\in W^{1,p}(\Lambda,\R^{d})$ with $p>d$, where $\Lambda$ is a bounded domain in $\R^{d}$.
 Moreover, assume that $f:\R^{d}\to \R$ is
 measurable and $E\subset\Lambda$ is measurable.
 We use the convention that $f(y(x))\abs{\det\nabla y(x)}=0$
 whenever $\abs{\det\nabla y(z)}=0$ for some $z\in E$
 and abbreviate $N_{y}(z,E)=\#(y^{-1}(z)\cap E)$ for any $z\in\R^{d}$,
where $\#$ denotes the counting measure.
 Then, if one of the functions $z\mapsto f(y(x))\abs{\det\nabla y(x)}$
 and $z\mapsto u(z)N_{f}(z,E)$ is integrable, so is the other
 one and the identity
 \begin{equation*}
 \int_{E} f(y(x))\abs{\det\nabla y(x)}\d x = \int_{\R^{m}} f(z)N_{f}(z,E) \d z
 \end{equation*}
holds. 
\end{lemma}

\begin{proposition}\label{prop:lscE}
For $p>d$ and $r>0$, the functional $\Ecal:W^{1,p}(\Omega;\R^d)\to [0,\infty]$ 
is
lower semicontinuous with repect to weak convergence in $W^{1,p}$.
\end{proposition}

\begin{proof}
The integrand of $\Ecal$ is polyconvex, since $(F,J)\mapsto|F|^p+J^{-r}$, $\R^{d\times d}\times (0,\infty)\to [0,\infty],$ is convex. As shown in detail by Ball \cite{Ba77a}, sequential weak lower semicontinuity of $\Ecal$ therefore follows from the weak continuity of the determinant, i.e., 
$y\mapsto \det \nabla y$ as a map between $W^{1,p}(\Omega;\R^d)$ and $L^{p/d}(\Omega)$, where both spaces are endowed with their weak topologies. 
\end{proof}

The Ciarlet--Ne\v cas condition is a viable constraint for direct methods:

\begin{lemma}[Weak stability of~\eqref{eq:CNC} {\cite[p.~185]{CiaNe87a}}]
\label{lem:stable}
 Let $y_{k}\rightharpoonup y_{\infty}$ in $W_{+}^{1,p}(\Omega,\R^{n})$, $p>d$,
 and assume that~\eqref{eq:CNC} holds for all $y_{k}$, $k\in\N$.
 Then~\eqref{eq:CNC} applies to $y_{\infty}$ as well.
\end{lemma}

Using the theory of maps of bounded distortion, we can obtain even more. For sufficiently large $p$ and $r$,
deformations with finite energy are open and discrete (see below) due to a result of Villamor and Manfredi \cite{ViMa98a}. 
With added global topological information, say, in form of \eqref{eq:CNC}, finite energy maps are even necessarily homeomorphisms \cite[Section 3]{GraKruMaiSte19a} (see also \cite{Kroe20a} for related results). In summary, we have the following.

\begin{proposition}\label{prop:homeo}
 Let $d\geq 2$,
 \[ p>d(d-1)\quad(\ge d),\qquad r>\frac{p(d-1)}{p-d(d-1)}\quad(>d-1),
 \]
and let $y\in W_{+}^{1,p}(\Omega,\R^{d})$ such that $\Ecal(y)<\infty$.
Then 
the continuous representative of $y$ is open (i.e., $y$ maps open subset of $\Omega$ to open sets in $\R^{d}$) and discrete (i.e., for each $z\in\R^d$, $y^{-1}(\{z\})$ does not have accumulation points in~$\Omega$). In particular, $y(\Omega)$ is open in $\R^d$.
If, in addition, \eqref{eq:CNC} holds, then $y$ is a homeomorphism on~$\Omega$ and $y^{-1}\in W_{+}^{1,\sigma}(y(\Omega),\Omega)$,
where
\[ \sigma:=\frac{(r+1)p}{r(d-1)+p}>d. 
\]
\end{proposition}

\begin{remark}\label{rem:homeo}
Possible self-contact on $\partial\Omega$ is not ruled out, and so 
$y$ is not necessarily a homeomorphism on~$\overline\Omega$.
\end{remark}

\begin{proof}[Proof of Proposition~\ref{prop:homeo}]
With $F=\nabla y(x)$,
\[	
	K^O(F):=\frac{\abs{F}^d}{\det F}
\]
is the \emph{outer distortion} of $y$ at $x$ (or \emph{dilatation} in the terminology of \cite{ViMa98a}). 
We infer from Young's inequality for some $\kappa,\rho\in(1,\infty)$ that
\begin{align*}
  \abs{K^O(F)}^{\kappa} \leq C_{\rho}\left(\abs F^{d\kappa\rho} + {\br{\frac1{\det F}}^{\frac{\kappa\rho}{\rho-1}}}\right).
\end{align*}
Given $y\in W_{+}^{1,p}(\Omega,\R^{d})$,
we find that $K^O(\nabla y)\in L^{\kappa}(\Omega)$ provided
$\Ecal(y)<\infty$ as well as $d\kappa\rho=p$ and
$\kappa=r(1-\frac1\rho)$.
Since $\frac1\kappa=\frac1r+\frac dp<\frac{p-d(d-1)}{p(d-1)} + \frac dp = \frac1{d-1}$
and $\rho=1+\frac p{dr}>1$
we may conclude
that $y$ is open and discrete as shown by Villamor and Manfredi \cite[Theorem 1]{ViMa98a}.

Finally, \eqref{eq:CNC} implies that $y\in W_{+}^{1,p}$ is a map of (Brouwer's) degree $1$ for values of its image
($y\in \rm{DEG}1$ by \cite[Remark 2.19(b)]{Kroe20a}, e.g.). 
By \cite[Thm.~6.8]{Kroe20a}, it now follows that $y:\Omega\to y(\Omega)$ is a homeomorphism with weakly differentiable inverse, and 
$\nabla (y^{-1})\in L^d(y(\Omega);\R^{d\times d})$.

Now that
$y$ is invertible with weakly differentiable inverse,
we may improve the last conclusion to $L^{\sigma}$.
To this end, we first apply a change of variables and use 
$F^{-1}=\frac{\cof F}{\det F}$, whence $|F^{-1}|\leq c\abs{F}^{d-1}|\det F|^{-1}$.
Then the assertion follows again by invoking Young's inequality to bound $\abs{F}^{(d-1)\sigma}|\det F|^{-(\sigma-1)}$.
\end{proof}

The final two lemmas of this section will be crucial ingredients in the construction of a recovery sequence in the proof of Theorem~\ref{thm:main}. They are also used in \cite{KroeVa19a,KroeVa22Pa} in similar fashion. For a closely related result and further references we refer to \cite[Theorem 5.1]{BaZa17a}.
\begin{lemma}[domain shrinking]\label{lem:shrinking}
Let $\Omega\subset \R^d$ be a bounded Lipschitz domain. Then there exists a sequence of $C^\infty$-diffeomorphisms 
\[
    \Psi_j:\overline{\Omega}\to \Psi_j(\overline{\Omega})\subset\subset \Omega
\]
such that as $j\to\infty$, $\Psi_j\to \identity$ in $C^m(\overline{\Omega};\R^d)$ for all $m\in \N$.
\end{lemma}
\begin{lemma}[composition with domain shrinking is continuous]\label{lem:shrinkcont}
Let $\Omega\subset \R^d$ be a bounded Lipschitz domain, $k\in \N_0$, $1\leq p <\infty$ and $f\in W^{k,p}(\Omega;\R^m)$, $m\in \N$.
With the maps $\Psi_j$ of Lemma~\ref{lem:shrinking}, we then have that $f\circ \Psi_j\to f$ in $W^{k,p}(\Omega;\R^n)$.
\end{lemma}
\begin{proof}[Proof of Lemma~\ref{lem:shrinking}]
If $\Omega$ is strictly star-shaped with respect to a point $x_0\in\Omega$, one may take 
$\Psi_j(x):=x_0+\frac{j-1}{j}(x-x_0)$. For a general Lipschitz domain, one can combine local constructions near the boundary using a smooth decomposition of unity: If, locally in some open cube $Q$, the set $\Omega$ is given as a Lipschitz subgraph, i.e.,
$\Omega\cap Q=\{x\in Q\mid x\cdot e\leq f(x')\}$ and $\partial\Omega\cap Q=\{x'+ef(x')\mid x\in Q\}$,
where $e$ is a unit vector orthogonal to one of the faces of $Q$, $x':=x-(x\cdot e) e$ and $f$ is a Lipschitz function,
we define
\[
    \hat{\Psi}_j(x;Q):=x'+\alpha_0 e+\frac{j-1}{j}(e\cdot x-\alpha_0)e\quad\text{for $x\in Q$, with $\alpha_0:=\inf_{x\in Q}e\cdot x$}.
\]
Notice that $\hat{\Psi}_j(\cdot;Q)$ pulls the local boundary piece $\partial\Omega\cap Q$ "down" (in direction $-e$) into the original domain while leaving the "lower" face of $Q$ fixed. Since $\partial\Omega$ can be covered by finitely many such cubes, we can write $\overline{\Omega}\subset Q_0\cup \bigcup_{k=1}^{n} Q_k$ with some open interior set $Q_0\subset\subset \Omega$. For a smooth decomposition of unity $1=\sum_{k=0}^n \varphi_k$ subordinate to this covering of $\Omega$ (i.e., $\varphi_k$ smooth, non-negative and compactly supported in $Q_k$), 
\[
    \Psi_j(x):=\varphi_0(x)x+\sum_{k=1}^n \varphi_j(x) \hat{\Psi}_j(x;Q_k)
\]
now has the asserted properties.
\end{proof}
\begin{proof}[Proof of Lemma~\ref{lem:shrinkcont}]
We only provide a proof for the case $k=1$, which will include the argument for $k=0$. For $k\geq 2$, the assertion follows inductively.
It suffices to show that as $j\to\infty$, $\partial_n [f\circ \Psi_j-f] \to 0$ in $L^p$, 
for each partial derivative $\partial_n$, $n=1,\ldots, d$. By the chain rule, 
\begin{align}\label{lemshrinkdef-1}
\begin{aligned}
	&\partial_n [f\circ \Psi_j-f]=
	[(\nabla f)\circ \Psi_j] \partial_n \Psi_j -\partial_n f \\
	&\quad =\Big([(\nabla f)\circ \Psi_j] \partial_n \Psi_j-(\partial_n f)\circ \Psi_j \Big)
	+\Big((\partial_n f)\circ \Psi_j -\partial_n f\Big).
\end{aligned}
\end{align}
The first term above converges to zero in $L^p$ since 
$(\nabla f) e_n=\partial_n f$ for the $n$-th unit vector $e_n$, and
$\partial_n\Psi_j \to \partial_n\identity=e_n$ uniformly. The convergence of the second term corresponds to our assertion for the case $k=0$, with 
$\tilde{f}:=\partial_n f\in L^p$. It can be proved in the same way as the well-known continuity of the shift in $L^p$:
If $\tilde{f}$ is smooth and can be extended to a smooth function on $\R^d$, we have
\begin{align}\label{lemshrinkdef-2}
	\norm{\tilde{f}\circ \Psi_j -\tilde{f}}_{L^p(\Omega;\R^m)}\leq \norm{\nabla\tilde{f}}_{L^\infty(\R^d;\R^{m\times d})} \norm{\Psi_j-\identity}_{L^p(\Omega;\R^d)}\underset{j\to\infty}{\longrightarrow} 0.
\end{align}
The general case follows by approximation of $\tilde{f}$ in $L^p$ with such smooth functions, by first extending 
$\tilde{f}$ by zero to all of $\R^d$, and then mollifying.
Here, notice that for the mollified function, $\|\nabla\tilde{f}\|_{L^\infty}$ in \eqref{lemshrinkdef-2} is unbounded in general as a function of the mollification parameter, but one can always choose the latter to converge slow enough with respect to $j$ so that \eqref{lemshrinkdef-2} still holds.
\end{proof}

\section{Elasticity with vanishing nonlocal self-repulsion}\label{sec:main}

In this section, we will study energies of the form
\[ E_{\eps} (y) = \begin{cases}
\Ecal(y) + \eps \Dcal(y) & \text{if }y\in W_{+}^{1,p}(\Omega,\R^{d}), \\
+\infty &\text{else},
\end{cases}
\]
in the limit $\eps\to 0^+$, in the sense of $\Gamma$-convergence with respect to the weak topology of $W^{1,p}$. 
Here, we say that $E_{\eps}$ $\Gamma$-converges to a functional $E_0$ if the following two properties hold
for every sequence $\eps(k)\to 0^+$ and every
$y\in W^{1,p}(\Omega;\R^d)$:
\begin{itemize}
\item[(i)] (lower bound) For all sequences $y_k\rightharpoonup y$ weakly in $W^{1,p}$, 
\[ \liminf_k E_{\eps(k)}(y_k)\geq E_0(y). \]
\item[(ii)] (recovery sequence) 
There exists a sequence $y_k\rightharpoonup y$ weakly in $W^{1,p}$ s.t. 
\[ \limsup_k E_{\eps(k)}(y_k)\leq E_0(y). \]
\end{itemize}
\begin{remark}
Notice that we do not require \emph{compactness} here, i.e., that any sequence $(y_k)$ with 
bounded $E_{\eps(k)}(y_k)$ has a subsequence weakly converging in $W^{1,p}$. This is automatic as soon as bounded energy implies a bound in the norm of $W^{1,p}$. However, in the most basic form, the energies we study are translation invariant and only control $\nabla y$ but not $y$. 
Of course, this would change as soon as a Poincar\'e inequality can be used due to a suitable boundary condition or other controls on $y$ or its average added via constraint or additional energy terms of lower order. 
\end{remark}

We will discuss three different examples for $\Dcal$, each preventing self-interpenetration, i.e., a loss of injectivity of $y$, in a different way.
Recall that
\[ E_{0}(y) =
\begin{cases}
\Ecal(y) & \text{if $y\in W_{+}^{1,p}(\Omega,\R^{d})$ satisfies~\eqref{eq:CNC},} \\
+\infty  & \text{else.}
\end{cases} \]

Throughout this section, we assume that
$d,p,r,\sigma$
satisfy the assumptions of Proposition~\ref{prop:homeo}, namely
\begin{equation}\label{eq:dprs}
 d\geq 2,\quad
 p>d(d-1),\quad
 r>\frac{p(d-1)}{p-d(d-1)}, \quad
 \sigma=\frac{(r+1)p}{r(d-1)+p}.
\end{equation}

\subsection{Bulk self-repulsion}\label{ssec:mainbulk}

Here we consider the energy $E_\eps$ with $\Dcal := \Dcal_{\Omega}$, i.e.,
\[ \Dcal(y) = \Dcal_{\Omega}(y) = \iint_{\Omega\times\Omega}\frac{\abs{x-\tilde x}^{q}}{\abs{y(x)-y(\tilde x)}^{d
+sq}}\abs{\det \nabla y(x)}\abs{\det \nabla y(\tilde x)}
\d x\d\tilde x \]
for $y\in W^{1,p}(\Omega,\R^{d})$, $q\in[1,\infty)$, $s\in[0,1)$.

The following statement is actually not required for our
main result.
However, together with its counterpart for the elastic energy (cf.~Proposition~\ref{prop:lscE})
it ensures well-posedness of the variational model.

\begin{proposition}\label{prop:lscD1}
For $p>d$, the functional
$\Dcal_\Omega:W^{1,p}(\Omega;\R^d)\to [0,\infty]$ 
is
lower semicontinuous with respect to weak convergence in $W^{1,p}$.
\end{proposition}

\begin{proof}
For $\delta>0$ define
\[
	\Dcal^{[\delta]}(y):=\iint_{\Omega\times\Omega}\frac{\abs{x-\tilde x}^{q}}{\max\{\delta,\abs{y(x)-y(\tilde x)}^{d+sq}\}}\abs{\det \nabla y(x)}\abs{\det \nabla y(\tilde x)}\,dx\,d\tilde x\leq \Dcal(y).
\]
Let $(y_k)\subset W^{1,p}(\Omega;\R^d)$ with $y_k\rightharpoonup y$ in $W^{1,p}$, for some 
$y\in W^{1,p}(\Omega;\R^d)$. In particular, $y_k\to y$ in $C(\overline\Omega;\R^d)$ by embedding, and consequently,
\begin{align}\label{wlscD-1}
	\liminf_{k\to\infty}\Dcal^{[\delta]}(y_k)=\liminf_{k\to\infty}
	\iint_{\Omega\times\Omega} W_y(x,\tilde x,\det \nabla y_k(x),\det \nabla y_k(\tilde x))\,dx\,d\tilde x
\end{align}
where 
\[
	W_y(x,\tilde x,J,\tilde{J}):=\frac{\abs{x-\tilde x}^{q}}{\max\{\delta,\abs{y(x)-y(\tilde x)}^{d+sq}\}}\abs{J}|\tilde{J}|
	\quad \text{for}~~x,\tilde x\in \Omega,~~J,\tilde{J}\in \R.
\]
Clearly, $W_y$ is symmetric in $(x,\tilde x)$ and $(J,\tilde{J})$, as well as separately convex in $(J,\tilde{J})$, i.e., convex in $J$ with $x,\tilde x,\tilde{J}$ fixed and convex in $\tilde{J}$ with $x,\tilde x,J$ fixed.
By \cite[Theorem 2.5]{Pe16a} (see also the related earlier result \cite[Theorem 11]{El11Pa}), this implies weak lower semicontinuity of $J\mapsto \iint_{\Omega\times\Omega} W_y(x,\tilde x,J(x),J(\tilde{x}))\,dx\,d\tilde x$ in $L^\alpha(\Omega)$, in particular for $\alpha:=\frac{p}{d}$. 
Again exploiting the weak continuity of the determinant, i.e., $J_k:=\det \nabla y_k\rightharpoonup J:=\det \nabla y$ weakly in $L^{p/d}$, we thus get that
\begin{align}\label{wlscD-2}
\begin{aligned}
  \liminf_{k\to\infty}\iint_{\Omega\times\Omega} W_y(x,\tilde x,\det \nabla y_k(x),\det \nabla y_k(\tilde x))\,dx\,d\tilde x & \\
	\quad \geq \iint_{\Omega\times\Omega} W_y(x,\tilde x,\det \nabla y(x),\det \nabla y(\tilde x))\,dx\,d\tilde x 
	&=	\Dcal^{[\delta]}(y).
\end{aligned}
\end{align}
Combining \eqref{wlscD-1} and \eqref{wlscD-2}, we see that $\Dcal^{[\delta]}$ is weakly lower semicontinuous for each $\delta>0$.
Since $\Dcal_\Omega(y)=\sup_{\delta>0} \Dcal^{[\delta]}(y)$, this implies weakly lower semicontinuity of~$\Dcal_\Omega$.
\end{proof}

\begin{theorem}\label{thm:main}
Let $\Omega\subset \R^d$, be a bounded Lipschitz domain, and suppose that
 $q\geq 1$ and $s\in[0,1)$. In addition, suppose that~\eqref{eq:dprs} holds together with
\begin{equation}\label{eq:sigma}
 s-\frac dq\le 1-\frac d\sigma.
\end{equation}
Then the functionals $E_{\eps}$ $\Gamma$-converge to $E_{0}$
 as $\eps\searrow 0$, with respect to the weak topology of $W^{1,p}(\Omega,\R^{d})$.
\end{theorem}

For the proof, we additionally need the following two propositions.

\begin{proposition}[Finite $E_\eps(y)$ implies \eqref{eq:CNC}]\label{prop:CNC}
Suppose that $s,q\geq 0$ and \eqref{eq:dprs} holds,
and let $y\in W_{+}^{1,p}(\Omega,\R^{d})$ such that $\Ecal(y)<\infty$ and $\Dcal(y)<\infty$. 
Then $y$ satisfies \eqref{eq:CNC}.
\end{proposition}
\begin{proof}
The proof is indirect. Suppose that \eqref{eq:CNC} does not hold. By the area formula (cf.~Lemma~\ref{lem:transform}),
this means that $Z_2:=\{z\in \R^d\mid N_y(z,\Omega)\geq 2\}$ has positive measure.
As a consequence, $X_2:=y^{-1}(Z_2)$ also has positive measure, because $y$ satisfies Lusin's property (N) as a map in $W^{1,p}$ with $p>d$. In addition, we claim that $X_2$ is open.
For a proof, take any $x\in X_2\neq \emptyset$. By definition of $X_2$, there exists another point $\tilde{x}\in X_2\setminus \{x\}$ such that $y(x)=y(\tilde{x})$. If we choose disjoint open neighborhoods $U,\tilde{U}\subset \Omega$ of $x,\tilde{x}$, respectively,
then $y(U)$ and $y(\tilde{U})$ are open sets because $y$ is open by Proposition~\ref{prop:homeo}. Hence, their intersection $y(U)\cap y(\tilde{U})\subset Z_2$ is also open, and it contains $y(x)=y(\tilde{x})$. By continuity of $y$, we conclude that 
$y^{-1}(y(U)\cap y(\tilde{U}))$ is now an open subset of $X_2$ containing $x$ (and $\tilde{x}$).
The above construction in particular shows that we can have two open, nonempty sets $V,W\subset X_2\subset \Omega$ with
$x\in V\subset U\cap y^{-1}(y(U)\cap y(\tilde{U}))$ and $W:=\tilde{U}\cap y^{-1}(y(V))$ 
such that $\overline V\cap \overline W=\emptyset$ and $y(W)\subset y(V)$ are open. Hence, with
\[
	\delta:=\min\{ \nabs{x-\tilde{x}} \mid x\in \overline V,~\tilde{x}\in \overline W\}>0,
\]
we have that
\[
	\Dcal(y)\geq \iint_{V\times W}\frac{{\delta}^{q}}{{\nabs{y(x)-y(\tilde{x})}}^{d+sq}}
\nabs{\det\nabla y(x)}\nabs{\det\nabla y(\tilde{x})}
\,\d x\d\tilde x.
\]
Changing variables in both integrals using Lemma~\ref{lem:transform}, also using that
$N_y\geq 1$ on the image of $y$, we infer that
\[
	\Dcal(y)\geq {\delta}^{q}\int_{y(W)}\int_{y(V)} \frac{1}{\nabs{\xi-\tilde{\xi}}^{d+sq}}
\,\d\xi\d\tilde\xi.
\]
In the inner integral, for each $\tilde\xi\in y(W)$, there is always at least one 
singularity at $\xi=\tilde\xi\in y(W)\subset y(V)$, an interior point. Since $sq\geq 0$,
this implies that $\Dcal(y)=+\infty$, contradicting our assumption.
\end{proof}

\begin{proposition}\label{prop:distor}
 Let $y\in W_{+}^{1,p}(\Omega,\R^{d})$ be a homeomorphism
 $\Omega\to y(\Omega)$
 with $y^{-1}\in W^{1,\sigma}(y(\Omega),\Omega)$, $q\in[1,\infty)$, $s\in[0,1)$. In addition, suppose that \eqref{eq:dprs} and~\eqref{eq:sigma} hold. 
 If $\Omega'$ is open and $\Omega'\subset\subset\Omega$
 then $\Dcal_{\Omega'}(y|_{\Omega'})<\infty$.
\end{proposition}

\begin{proof}
 We apply Lemma~\ref{lem:transform} twice to change variables in each of the two integrals in $\mathcal D_{\Omega'}$, 
 with $E=\Omega'$. First, change variables in the inner integral, say, over $x$, using $z=\xi$ and 
$f(z)={\nabs{y^{-1}(z)-\tilde{x}}^{q}}{\nabs{z-y(\tilde{x})}^{-(d+sq)}}$ for any fixed $\tilde{x}\in \Omega'$. Afterwards, use Fubini's theorem to change the order of integration and change variables in the integral over $\tilde{x}$, now for any fixed $\xi$
using $z=\tilde{\xi}$ and $f(z)={\nabs{y^{-1}(\xi)-y^{-1}(z)}^{q}}{\nabs{\xi-z}^{-(d+sq)}}$.
We thus obtain that
 \begin{align*}
 \mathcal D_{\Omega'}(y|_{\Omega'})
 &=\iint_{y(\Omega')\times y(\Omega')}\frac{\nabs{y^{-1}(\xi)-y^{-1}(\tilde \xi)}^{q}}{\nabs{\xi-\tilde\xi}^{d+sq}}
\d\xi\d\tilde\xi.
 \end{align*}
 The right-hand side is just the $q$-th power of
 the seminorm belonging to the Sobolev--Slobodecki\u\i\
 space $W^{s,q}(y(\Omega'),\R^{d})$.
 As $\overline{y(\Omega')}\subset y(\Omega)$
 we may find $\psi\in C^{\infty}(\R^{d})$ supported in
 $y(\Omega)$ with $\psi=1$ on $y(\Omega')$.
 Choosing any regular value $r\in(0,1)$,
 the set $\psi^{-1}((r,1])$
 has a smooth boundary.
 We denote its component containing $y(\Omega')$ by~$\Upsilon$.
 In case $s\in(0,1)$, using  $y(\Omega')\subset\Upsilon\subset y(\Omega)$
 and applying the embedding theorem, we infer
 \begin{align*}
 \mathcal D_{\Omega'}(y|_{\Omega'})
 &\le \sq{y^{-1}|_{y(\Omega')}}_{W^{s,q}(y(\Omega'),\R^{d})}^{q}
 \le \sq{y^{-1}|_{\Upsilon}}_{W^{s,q}(\Upsilon,\R^{d})}^{q} \\
 &\le C_{d,q,\sigma,\Upsilon}\norm{y^{-1}|_{\Upsilon}}_{W^{1,\sigma}(\Upsilon,\R^{d})}^{q}
 \le C_{d,q,\sigma,\Upsilon}\norm{y^{-1}|_{y(\Omega)}}_{W^{1,\sigma}(\Omega,\R^{d})}^{q}.
 \end{align*}
The case $s=0$ is similar; in the intermediate step above, we now use $W^{\tilde{s},q}$ with some $\tilde{s}>0$ small enough so that $W^{1,\sigma}$ still embeds into $W^{\tilde{s},q}$ since $\sigma> \frac{dq} {q+d}$.
\end{proof}

\begin{proof}[Proof of Theorem~\ref{thm:main}]
\emph{Lower bound ($\mathit\Gamma$-lim\,inf-inequality):} Assume that $y_{k}\rightharpoonup y$ in $W^{1,p}$
 and $\eps_{k}\searrow 0$. 
If $\liminf_k E_{\eps_k}(y_k)=+\infty$, there is nothing to show.
Hence, passing to a suitable subsequence (not relabeled), we may assume that 
the $\liminf$ is a limit and $E_{\eps_k}(y_k)$ is bounded.
Since
 $\Dcal\ge0$ and $\Ecal$ is weakly lower semicontinuity, we get that
\begin{align}\label{thmmain-wlsc}
	\lim_k E_{\eps_k}(y_k)\geq \liminf_k \Ecal(y_k)\geq \Ecal(y).
\end{align}
Moreover, by Proposition~\ref{prop:CNC}, we see 
that $y_k$ satisfies~\eqref{eq:CNC} for all $k$, and by Lemma~\ref{lem:stable}, this implies that
$y$ satisfies~\eqref{eq:CNC}. Hence, $\Ecal(y)=E_0(y)$, and \eqref{thmmain-wlsc} thus implies the asserted lower bound.
 
\emph{Upper bound (construction of a recovery sequence):} 
Let $y\in W^{1,p}(\Omega;\R^d)$ be given. We may assume that $E_0(y)<\infty$, because otherwise there is nothing to show. We therefore have that
$y\in W_{+}^{1,p}(\Omega;\R^d)$, $y$ satisfies \eqref{eq:CNC} and $\Ecal(y)<\infty$.
 
By Proposition~\ref{prop:homeo}, the map $y:\Omega\to y(\Omega)$ is a homeomorphism.
 We choose $j\in\N$ sufficiently large and shrink the domain $\Omega$ to $\Omega_{j}=\Psi_j(\Omega)$, 
using Lemma~\ref{lem:shrinking}.
Now define
\[ y_{j}:\Omega\to y(\Omega_{j}), \qquad y_{j}:=y|_{\Omega_{j}}\circ \Psi_j. \]
As $j\to \infty$, $y_{j}\to y$ in $W^{1,p}(\Omega;\R^d)$ 
and $\Ecal(y_{j})\to\Ecal(y)$ by Lemma~\ref{lem:shrinkcont}. Here, 
concerning the term $(\det\nabla y)^{-r}$ in $\Ecal$, notice that 
by the chain rule and the muliplicativity of the determinant,
we have that 
\[
\frac{1}{(\det\nabla y_j)^{r}}
=(f\circ\Psi_j) \frac{1}{(\det\nabla \Psi_j)^{r}},\quad\text{with}~~~
f:=\frac{1}{(\det\nabla y)^{r}}\in L^1(\Omega).
\]
Combined with the fact that $\det\nabla \Psi_j\to 1$ uniformly,
Lemma~\ref{lem:shrinkcont} can therefore indeed be applied with $k=0$ and $p=1$ to obtain convergence of this singular term in $\Ecal$. 

Since $\Omega_j\subset\subset \Omega$ for each $j$,
 we also have that
\[
	\Dcal_{\Omega}(y_j)=\Dcal_{\Omega_j}(y) <\infty
\]
 by change of variables and Proposition~\ref{prop:distor}.

 Now let $\eps_{k}\searrow0$ be given.
 We choose $j_{k}\to\infty$ such that
 $\eps_{k}\Dcal(y_{j_{k}})\xrightarrow{k\to\infty}0$.
 So
 $E_{\eps_{k}}(y_{j_{k}}) = \Ecal(y_{j_{k}}) + \eps_{k}\Dcal(y_{j_{k}})\xrightarrow{k\to\infty}
 \Ecal(y) = E_{0}(y)$. 
\end{proof}

\subsection{Bulk self-repulsion near the boundary}\label{ssec:mainbulkboundary}
Here, we consider $\Dcal := \Dcal_{U_{\delta}}$
where, for any $\delta>0$, the set $U_{\delta}\subset\Omega$
can be chosen as any open neighborhood of $\partial\Omega$
which is at least $\delta$-thick in the sense that
\begin{align}\label{eq:delta-thick} 
U_{\delta}\supset \br{\partial\Omega}^{(\delta)}
= \sett{x\in\Omega}{\dist(x,\partial\Omega)<\delta}. 
\end{align}
For the ease of notation, we abbreviate
$\Dcal_{\delta} := \Dcal_{U_{\delta}}$ so that
\[ \Dcal(y) = \Dcal_{\delta}(y) = 
\iint_{U_{\delta}\times U_{\delta}}\frac{\abs{x-\tilde x}^{q}}{\abs{y(x)-y(\tilde x)}^{d
+sq}}\abs{\det \nabla y(x)}\abs{\det \nabla y(\tilde x)}
\d x\d\tilde x \]
for $y\in W^{1,p}(\Omega,\R^{d})$, $q\in[1,\infty)$, $s\in[0,1)$.
The combined energy $E_\eps$ now also depends on the choice of $U_\delta$,  
and to make this more visible, we will now write
\[ E_{\eps,\delta} (y) = \begin{cases}
\Ecal(y) + \eps \Dcal_\delta(y) & \text{if }y\in W_{+}^{1,p}(\Omega,\R^{d}), \\
+\infty &\text{else}.
\end{cases}
 \]
We will see that $E_0$ is still the correct limit functional  
independently of $\delta$. In fact, we can even allow the simultaneous limit as 
$(\eps,\delta)\to (0,0)$. 
\begin{remark}
The fact that the limit as $\delta\to 0^+$ is admissible offers 
an attractive choice of $U_\delta$ for numerical purposes: a single boundary layer of the triangulation, which requires $\delta$ of the order of the grid size $h$. 
In that case, the cost of a numerical evaluation of $\Dcal_\delta$ scales like $h^{-2(d-1)}$ 
(like a double integral on the surface), which is much cheaper than for $\Dcal_\Omega$
which scales like $h^{-2d}$.
\end{remark}

As before, for fixed $\eps,\delta>0$, the functional $E_{\eps,\delta}$ is well suited for  
minimization by the direct method:  
\begin{proposition}\label{prop:lscD2}
For $p>d$, the functional
$\Dcal_\delta:W^{1,p}(\Omega;\R^d)\to [0,\infty]$ 
is lower semicontinuous with respect to weak convergence in $W^{1,p}$.
\end{proposition}
\begin{proof}                                                              
This is Proposition~\ref{prop:lscD1} with $\Omega$ replaced by $U_\delta$. 
Here, notice that no boundary regularity of $U_\delta$ is required: If needed, 
we can cover $U_\delta$ from inside with open domains with smooth boundary, 
and since the integrand of $\Dcal_\delta$ is nonnegative, we can therefore write
$\Dcal_\delta$ as a supremum of weakly lower semicontinuous funtionals 
using the smooth smaller domains as domain of integration. 
\end{proof}

\begin{theorem}\label{thm:main2}
Let $\Omega\subset \R^d$ be a bounded Lipschitz domain such that $\R^d\setminus \partial\Omega$ has exactly two connected components, $q\in[1,\infty)$, $s\in[0,1)$. In addition, suppose that \eqref{eq:dprs} and \eqref{eq:sigma} hold.
Then for any $\delta_0\in [0,\infty]$, the functionals 
$E_{\eps,\delta}$ $\Gamma$-converge to $E_{0}$
 as $(\eps,\delta)\to (0,\delta_0)$ ($\eps,\delta>0$), with respect to the weak topology of $W^{1,p}(\Omega,\R^{d})$.
\end{theorem}

For the proof, we additionally need the following modification of
Proposition~\ref{prop:CNC}.

\begin{proposition}[Finite $E_{\eps,\delta}(y)$ implies \eqref{eq:CNC}]\label{prop:CNC2}
Let $\Omega\subset \R^d$ be a bounded Lipschitz domain such that $\R^d\setminus \partial\Omega$ has exactly two connected components, $q\in[1,\infty)$, $s\in[0,1)$. In addition, suppose that \eqref{eq:dprs} holds.
Moreover, let $y\in W_{+}^{1,p}(\Omega,\R^{d})$ such that $\Ecal(y)<\infty$ and $\Dcal_\delta(y)<\infty$. 
Then $y$ satisfies \eqref{eq:CNC} on $\Omega$.
\end{proposition}

\begin{proof}
Analogously to Proposition~\ref{prop:CNC}, we infer that
$y$ satisfies \eqref{eq:CNC} on $U_\delta$. By Proposition~\ref{prop:homeo},
we obtain that $y:U_\delta\to y(U_\delta)$ is a homeomorphism. Here, notice that
for this conclusion, we do not need any regularity of the boundary of $U_\delta$, since it is enough
to apply Proposition~\ref{prop:homeo} with $\Omega$ replaced by subdomains of $U_\delta$ with smooth boundary, and a sequence of such subdomains covers $U_\delta$ from the inside. 

Such an inner covering can also be used for $\Omega$: Choose open $\Omega_{j}\nearrow\Omega$
such that $\partial\Omega_{j}$ is smooth, say, Lipschitz.
In addition, using that $\partial\Omega$ itself is also Lipschitz,
we can make sure that as for $\Omega$, we have that
$\R^d\setminus \partial\Omega_j$ has exactly two connected components.
For any fixed $\delta$, there exists a sufficiently large $j$ such that $\partial\Omega_{j}$ 
is contained in the open $\delta$-neighborhood $(\partial\Omega)^{(\delta)}$ of $\partial\Omega$, 
and therefore $\partial\Omega_{j}\subset U_\delta$.
Consequently, $y|_{\partial\Omega_j}$ is injective, 
which implies that $y\in \operatorname{AIB}(\Omega_j)$ in the sense of \cite[Def.~2.1 and 2.2]{Kroe20a}.
In addition, we know that $y\in W_+^{1,p}(\Omega_j;\R^d)$. 
By \cite[Thm.~6.1 and Rem.~6.3]{Kroe20a}, we infer that $y$ satisfies \eqref{eq:CNC} on $\Omega_j$.
As the latter holds for all $j$, we conclude that
$y$ satisfies \eqref{eq:CNC} on $\Omega$ by monotone convergence.
\end{proof}

\begin{remark}
The proof of Proposition~\ref{prop:CNC2} exploits that $y$ is a homeomorphism near the boundary, which we obtain from Proposition~\ref{prop:homeo}. This forces the relatively restrictive assumptions on $p$ and $r$. While this may be technical to some degree, some restrictions are definitely needed. In fact, by itself, \eqref{eq:CNC} on a boundary strip like $U_\delta$ is not strong enough to provide the necessary global topological information: If one can squeeze surfaces to points with a deformation of finite elastic energy (possible if $p$ and $r$ are small enough), then self-penetration on $U_\delta$ is indeed possible for a $y\in W_+^{1,p}$ which is injective on $U_\delta$ outside a set of dimension $d-1$. Such a set of $d$-dimensional measure zero is invisible to \eqref{eq:CNC}. 
\end{remark}

\begin{proof}[Proof of Theorem~\ref{thm:main2}]
\emph{Lower bound ($\mathit\Gamma$-lim\,inf-inequality):} 
This is completely analogous to the proof of Theorem~\ref{thm:main}, using 
Proposition~\ref{prop:CNC2} instead of Proposition~\ref{prop:CNC}.
 
\emph{Upper bound (construction of a recovery sequence):} 
Again, we can follow the proof of Theorem~\ref{thm:main} step by step, using the domain shrinking maps $\Psi_j$ to define $y_j:=y\circ \Psi_j$ as before. In particular, 
changing variables we now observe that
\[
	\Dcal_{\delta}(y_j)=\Dcal_{\Psi_j(U_\delta)}(y)
	\leq \Dcal_{\Omega_j}(y) <\infty
\]
by Proposition~\ref{prop:distor},
for any fixed $j$. Given $(\eps_k,\delta_k)\to (0,\delta_0)$, 
we thus again get a suitable recovery sequence given by $(y_{j(k)})$ as long as $j(k)\to\infty$ slow enough so that $\eps_k \Dcal_{\delta_k}(y_{j(k)})\to 0$. 
\end{proof}

\subsection{Surface self-repulsion}\label{ssec:mainsurface}

Here we look at $\Dcal:=\Dcalt_{\partial\Omega}$ where
\[ \Dcalt_{\partial\Omega}(y) = \iint_{\partial\Omega\times\partial\Omega}\frac{\abs{x-\tilde x}^{q}}{\abs{y(x)-y(\tilde x)}^{d-1
+sq}}
\d A(x)\d A(\tilde x), \qquad q\in[1,\infty), s\in[0,1), \]
and $A(\cdot)$ denotes the $(d-1)$-dimensional Hausdorff
measure. Again, this is a term well-suited for minimization via the direct method:

\begin{proposition}\label{prop:lscD3}
For $p>d$, the functional
$\Dcalt_{\partial\Omega}:W^{1,p}(\Omega;\R^d)\to [0,\infty]$ 
is
lower semicontinuous with respect to weak convergence in $W^{1,p}$.
\end{proposition}

\begin{proof}
 As the trace of $y\in W^{1,p}(\Omega,\R^{d})$ on $\partial\Omega$ is compactly embedded in~$L^{1}$,
 we obtain pointwise a.e.\ convergence of the integrand.
 So the claim immediately follows from Fatou's Lemma.
\end{proof}

We will employ results of~\cite{Kroe20a}
which require that $\Omega$ does not have ``holes''
as made precise in the following statement.

\begin{theorem}\label{thm:main3}
Let $\Omega\subset \R^d$ be a bounded Lipschitz domain
such that $\R^{d}\setminus\partial\Omega$ has exactly two connected components.
Given~\eqref{eq:dprs}, we require $q\geq 1$ and $s\in[0,1]$ to be chosen such that
\begin{equation}\label{eq:sigma3}
	q \br{(1-s)\sigma-d}>d^2-\sigma
\end{equation}
and
\begin{equation}\label{eq:self-rep}
 sq\ge (d-1)\frac{p+d}{p-d}.
\end{equation}
Then the functionals $E_{\eps}$ $\Gamma$-converge to $E_{0}$
 as $\eps\searrow 0$, with respect to the weak topology of $W^{1,p}(\Omega,\R^{d})$.
\end{theorem}

\begin{remark}
 With $\sigma=\sigma(r,p,d)>d$ as defined in \eqref{eq:dprs}, the conditions~\eqref{eq:sigma3} and~\eqref{eq:self-rep}
 are met if  $0< s<1-\frac d\sigma$ and $q>\max\left\{\frac{d^2-\sigma}{(1-s)\sigma-d},\frac{d-1}s\cdot\frac{p+d}{p-d}\right\}$.
\end{remark}

For the proof, we additionally need the following two propositions.

\begin{proposition}[Finite $E_\eps(y)$ implies \eqref{eq:CNC}]\label{prop:CNC3}
Let $\Omega\subset \R^d$ be a bounded Lipschitz domain
such that $\R^{d}\setminus\partial\Omega$ has exactly two connected components.
Suppose that~\eqref{eq:dprs} and~\eqref{eq:self-rep} hold.
Moreover, let $y\in W_{+}^{1,p}(\Omega,\R^{d})$ such that $\Ecal(y)<\infty$ and $\Dcalt_{\partial\Omega}(y)<\infty$. 
Then $y$ satisfies~\eqref{eq:CNC}.
\end{proposition}

\begin{proof}
 According to~\cite[Cor.~6.5]{Kroe20a}
 it is enough to show injectivity of $y|_{\partial\Omega}$.
 If the latter is not the case,
 we may choose $x_{0},\tilde x_{0}\in\partial\Omega$, $x_{0}\ne\tilde x_{0}$,
 such that $y(x_{0})=y(\tilde x_{0})$.
 Recalling that $y\in C^{0,\alpha}(\overline\Omega,\R^{d})$,
 $\alpha=1-\frac dp$,
 and abbreviating $\eps=\tfrac13\abs{x_{0}-\tilde x_{0}}$, $t=d-1+sq$, we infer
 \begin{align*}
  \Dcalt_{\partial\Omega}(y)
  &=\iint_{\partial\Omega\times\partial\Omega}\frac{\abs{x-\tilde x}^{q}}{\abs{y(x)-y(\tilde x)}^{t}} \d A(x)\d A(\tilde x) \\
  &\ge\iint_{(\partial\Omega\cap B_{\eps}(x_{0}))\times(\partial\Omega\cap B_{\eps}(\tilde x_{0}))}\frac{\abs{x-\tilde x}^{q}}{\br{\abs{y(x)-y(x_{0})}+\abs{y(\tilde x_{0})-y(\tilde x)}}^{t}} \d A(x)\d A(\tilde x) \\
  &\ge c_{\alpha}\iint_{(\partial\Omega\cap B_{\eps}(x_{0}))\times(\partial\Omega\cap B_{\eps}(\tilde x_{0}))}\frac{\br{\frac\eps3}^{q}}{\br{\abs{x-x_{0}}^{\alpha}+\abs{\tilde x_{0}-\tilde x}^{\alpha}}^{t}} \d A(x)\d A(\tilde x).
 \end{align*}
 Introducing local bi-Lipschitz charts $\Phi:V\to\partial\Omega\cap B_{\eps}(x_{0})$, $\tilde\Phi:\tilde V\to\partial\Omega\cap B_{\eps}(\tilde x_{0}))$
 where $V,\tilde V\subset\R^{d-1}$ are open sets
 and $\Phi(0)=x_{0}$, $\tilde\Phi(0)=\tilde x_{0}$,
 we arrive at
 \begin{align*}
  \Dcalt_{\partial\Omega}(y)
  &\ge c_{\alpha,\eps,q,t}\iint_{V\times\tilde V}\frac{\sqrt{\det D\Phi(\xi)^{\top}D\Phi(\xi)}\sqrt{\det D\tilde\Phi(\tilde\xi)^{\top}D\tilde\Phi(\tilde\xi)}}{\br{\abs{\Phi(\xi)-\Phi(0)}^{2}+\abs{\tilde\Phi(\tilde\xi)-\tilde\Phi(0)}^{2}}^{\alpha t/2}} \d\xi\d\tilde\xi \\
  &\ge c_{\alpha,\eps,q,t,\Phi}\iint_{V\times\tilde V}\frac{\d\xi\d\tilde\xi}{\br{\abs{\xi}^{2}+\abs{\tilde\xi}^{2}}^{\alpha t/2}}.
 \end{align*}
 By assumption, both $V$ and $\tilde V$ contain $B_{\delta}(0)\subset\R^{d-1}$
 for some $\delta>0$.
 Decomposing $(\xi^{\top},\tilde\xi^{\top})=\rho\eta^{\top}\in\R^{2d-2}$
 where $\rho>0$, $\eta\in\mathbb S^{2d-3}$, yields
 \begin{align*}
  \Dcalt_{\partial\Omega}(y)
  &\ge c_{\alpha,\eps,q,t,\Phi,d}\int_{0}^{\delta}\frac{\rho^{2d-3}}{\rho^{\alpha t}}\d\rho.
 \end{align*}
 This term is infinite provided $2d-3-\alpha t\le-1$
 which is equivalent to~\eqref{eq:self-rep}.
\end{proof}

\begin{proposition}\label{prop:distor3}
 Let $\Omega\subset \R^d$ be a bounded Lipschitz domain.
 Assume that $y\in W_{+}^{1,p}(\Omega,\R^{d})$ is a homeomorphism
 $\Omega\to y(\Omega)$
 with $y^{-1}\in W^{1,\sigma}(y(\Omega),\Omega)$
 for which~\eqref{eq:sigma3} applies.
 If $\Omega'$ is open and $\Omega'\subset\subset\Omega$
 and $\partial\Omega'$ is Lipschitz,
 then $\Dcalt_{\partial\Omega'}(y)<\infty$.
\end{proposition}

\begin{proof}
First notice that $y(\Omega)$ is open and bounded in $\R^{d}$, the former by invariance of domain (see \cite[Theorem 3.30]{FoGa95B}, e.g.) and the latter due to the fact that $y\in C(\overline{\Omega};\R^{d})$ by embedding.
Hence, $y(\overline{\Omega'})$ is a compact and connected subset of $y(\Omega)$ with positive distance to $\partial [y(\Omega)]$.
We choose a domain $\Lambda\subset \R^{d}$ with smooth boundary such that $y(\Omega')\subset\subset \Lambda \subset\subset y(\Omega)$.
By embedding, $y^{-1}\in C^{0,\beta}(\Lambda,\R^{d})$, $\beta=1-\frac d\sigma$.
Abbreviating $t=d-1+sq$, we arrive at
 \begin{align*}
  \Dcalt_{\partial\Omega'}(y)
  &=\iint_{\partial\Omega'\times\partial\Omega'}\frac{\abs{x-\tilde x}^{q}}{\abs{y(x)-y(\tilde x)}^{t}} \d A(x)\d A(\tilde x) \\
  &\le C_{\beta}\iint_{\partial\Omega'\times\partial\Omega'}\abs{x-\tilde x}^{q-t/\beta} \d A(x)\d A(\tilde x).
 \end{align*}
 The term $\abs{x-\tilde x}$ is bounded above since $\Omega$ is bounded.
 It approaches zero only in a neighborhood of the diagonal.
 In order to show that $\Dcalt_{\partial\Omega'}(y)$ is finite
 we only have to consider
 $\iint_{\Phi(V)\times\Phi(V)}\abs{x-\tilde x}^{q-t/\beta} \d A(x)\d A(\tilde x)$
 where $\Phi:V\to U\subset\partial\Omega'$ is a chart
 and $V\subset B_{R}(0)\subset\R^{d-1}$ is an open set.
 Decomposing $\xi=\rho\eta\in\R^{d-1}$
 where $\rho>0$, $\eta\in\mathbb S^{d-2}$, yields
 \begin{align*}
 &\iint_{\Phi(V)\times\Phi(V)}\abs{x-\tilde x}^{q-t/\beta} \d A(x)\d A(\tilde x) \\
 &=\iint_{V\times V}\abs{\Phi(\xi)-\Phi(\tilde\xi)}^{q-t/\beta}
 {\sqrt{\det D\Phi(\xi)^{\top}D\Phi(\xi)}\sqrt{\det D\Phi(\tilde\xi)^{\top}D\Phi(\tilde\xi)}}
 \d\xi\d\tilde\xi \\
 &\le C_{\Phi}\iint_{V\times V}\abs{\xi-\tilde\xi}^{q-t/\beta}
 \d\xi\d\tilde\xi \\
 &\le C_{\Phi}\int_{B_{R}(0)}\int_{B_{R}(0)}\abs{\xi-\tilde\xi}^{q-t/\beta}
 \d\xi\d\tilde\xi \\
 &\le C_{\Phi}\int_{B_{3R}(0)}\int_{B_{R}(\tilde\xi)}\abs{\xi}^{q-t/\beta}
 \d\xi\d\tilde\xi \\
 &\le C_{\Phi,d,R}\int_{0}^{R}\rho^{d-1+q-t/\beta}\d\rho.
 \end{align*}
 The right-hand side is finite if $d-1+q-t/\beta>-1$ 
 which is equivalent to~\eqref{eq:sigma3}.
\end{proof}

\begin{proof}[Proof of Theorem~\ref{thm:main3}]
We proceed as in the proof of Theorem~\ref{thm:main}.
For the \emph{lower bound} we use
Proposition~\ref{prop:CNC3} in place of
Proposition~\ref{prop:CNC}.
To see that the \emph{recovery sequence} also works for $\Dcalt_{\partial\Omega}$,
we compute
 \begin{align*}
  \Dcalt_{\partial\Omega}(y_{j})
  &=\Dcalt_{\partial\Omega}(y|_{\Omega_{j}}\circ \Psi_j) \\
  &=\iint_{\partial\Omega\times\partial\Omega}\frac{\abs{\xi-\tilde\xi}^{q}}{\abs{y(\Psi_{j}(\xi))-y(\Psi_{j}(\tilde\xi))}^{d-1+sq}} \d A(\xi)\d A(\tilde\xi) \\
  &\le C_{\Psi_{j}}\iint_{\partial\Omega_j\times\partial\Omega_j}\frac{\abs{\Psi_{j}^{-1}(x)-\Psi_{j}^{-1}(\tilde x)}^{q}}{\abs{y(x)-y(\tilde x)}^{d-1+sq}} \d A(x)\d A(\tilde x) \\
  &\le C_{\Psi_{j}}\norm{\nabla\Psi_{j}^{-1}}_{L^{\infty}}^{q}\iint_{\partial\Omega_j\times\partial\Omega_j}\frac{\abs{x-\tilde x}^{q}}{\abs{y(x)-y(\tilde x)}^{d-1+sq}} \d A(x)\d A(\tilde x)
 \end{align*}
where $C_{\Psi_{j}}$ denotes a factor that bounds the
terms arising from the change of variables.
Now we deduce from Proposition~\ref{prop:distor3}
(instead of Proposition~\ref{prop:distor})
that the right-hand side is finite.
\end{proof}

\subsection{Further generalizations and remarks}\label{ssec:general}

\begin{remark}[More general elastic energies]
It is easy to see that throughout, the integrand of $\Ecal$ can be replaced by any polyconvex function admitting the original integrand as a lower bound (up to multiplicative and additive constants). Moreover, the latter is only exploited for the application of the theory of functions of bounded distortion in Proposition~\ref{prop:homeo}. More precisely, Theorems~\ref{thm:main}, \ref{thm:main2} and \ref{thm:main3} also hold for any elastic energy of the form
\[ \Ecal(y) = \int_{\Omega} W(\nabla y(x))\,\d x \]
such that
\begin{enumerate}
\item[(i)] $W:\R^{d\times d}\to (-\infty,+\infty]$ is continuous and polyconvex,
\item[(ii)] $W(F)\geq c |F|^p-C$ for all $F\in \R^{d\times d}$, where $p>d$,
\item[(iii)] $W(F)\geq c \left(\frac{|F|^{d}}{\det F}\right)^\beta-C$ 
for all $F\in \R^{d\times d}$ with $\det F>0$, where $\beta>d-1$.
\end{enumerate}
Here, $p>d,\beta>d-1$, $c>0$, and $C\in\R$ are constants. Notice that (iii) directly provides the bound on 
the outer distortion we need to generalize Proposition~\ref{prop:homeo}.
\end{remark}

\begin{remark}[Boundary conditions and force terms]
Due to the stability of $\Gamma$-convergence with respect to addition of continuous functionals, our main results continue to hold
if $\Ecal$ is modified by adding a term which is continuous in the weak topology of $W^{1,p}$ 
(typically either linear or lower order, exploiting a compact embedding). This includes many classical force potentials for body forces and surface tractions. 
Additional boundary conditions, say, a Dirichlet condition of the form $y=y_0$ on a part $\Lambda$ of $\partial\Omega$,
are in principle also possible but not trivial to add, as they require modified recovery sequences in the proof of the theorems. In particular, we would need a suitable modification of Lemma~\ref{lem:shrinking} which keeps the Dirichlet part of the boundary fixed, as well as additional assumptions on $y_0$ which at the very least should map $\overline{\Lambda}$ to a reasonably smooth set out of self-contact. The easiest way to set up a meaningful model with full coercivity in $W^{1,p}$ which is compatible with our theorems is to confine the deformed material to a box by constraint ($y(\Omega)\subset \mathcal{B}$ for a given compact $\mathcal{B}\subset \R^d$ with non-empty interior).
\end{remark}

\begin{remark}[More general nonlocal self-repulsive terms]
It is clear that our general proof strategy can also be applied to other nonlocal terms $\Dcal$. 
The only key features of such a term $\Dcal$ are the following:
\begin{itemize}
\item[(i)] for any deformation $y$ with finite elastic energy $\Ecal(y)$,
finite $\Dcal(y)$ implies \eqref{eq:CNC} (cf.~Propositions~\ref{prop:CNC}, \ref{prop:CNC2} and \ref{prop:CNC3});
\item[(ii)] for any homeormorphisms $y\in W^{1,p}_+$ ($p>d$) whose inverse has the Sobolev regularity $W^{1,\sigma}$ ($\sigma>d$) obtained from the control of its distortion through the elastic energy (see Proposition~\ref{prop:homeo}), we obtain
$\Dcal(y)<\infty$, at least if we move to a slightly smaller domain $\Omega'\subset\subset \Omega$
(cf.~Propositions~\ref{prop:distor} and \ref{prop:distor3}).
\end{itemize}
Moreover, following \cite{KroeVa19a,KroeVa22Pa}, it is in principle possible to work with an everywhere finite 
$\Dcal_\eps$ instead of $\eps \Dcal$ (say, a suitable truncation of the latter), if we restrict ourselves to deformations satisfying a fixed energy bound. 
Here, the basic idea is to find at least one deformation $y_0$ so that $e_0:=\Ecal(y_0)+\sup_{0<\eps\leq 1}\Dcal_\eps(y_0)<\infty$,
for instance the identity or another map far from self-contact.
Then check if (i) still holds in such a case if we replace the assumption $\Dcal(y)<\infty$ by $\Ecal(y)+\Dcal_\eps(y)\leq e_0$ (for sufficiently small $\eps$ independently of $y$).
\end{remark}

\begin{remark}[Mosco-covergence and recovery by homeomorphisms]
Our proofs of Theorems~\ref{thm:main}, \ref{thm:main2} and \ref{thm:main3} actually provide more than $\Gamma$-convergence: The recovery sequence we construct always converges strongly in $W^{1,p}$, which means that we actually proved so called Mosco-convergence. Moreover, as constructed, each member of the recovery sequence is a homeomorphism on $\overline{\Omega}$. 
In particular, any admissible $y$ with finite $E_0(y)$ is always contained in the $C^0$-closure of these homeomorphisms, i.e., $y\in AI(\overline{\Omega})$ in the notation of \cite{Kroe20a}. 
Our results here therefore also show that
within $W^{1,p}_+(\Omega;\R^d)$ with $p>d$, $AI(\overline{\Omega})$
coincides the class of maps satisfying $\eqref{eq:CNC}$ if we also impose strong enough a-priori bounds on the outer distortion to apply the result of Villamor and Manfredi as in Proposition~\ref{prop:homeo}.
The general case is still not clear, cf.~\cite[Remark 2.19]{Kroe20a}.
\end{remark}

\subsection*{Acknowledgements}The work of S.K.~was supported by the GA \v{C}R-FWF grant 19-29646L.
Major parts of this research were carried out during mutual research visits
of S.K. at the Chemnitz University of Technology and of 
Ph.~R. at \'{U}TIA, whose hospitality is gratefully acknowledged.

\end{document}